\documentclass[a4paper,fleqn,10pt,draft]{scrartcl}
\usepackage{amsmath}
\usepackage{amsthm}
\usepackage{amsfonts}
\usepackage{url}
\usepackage{color}
\usepackage[affil-it,blocks]{authblk}
\usepackage{enumerate}

\usepackage{accents}

\newcommand{\dvol}{\,\mathrm{d}V}
\newcommand{\Scal}{\operatorname{Sc}}
\newcommand{\Rm}{\operatorname{Rm}}

\newtheorem{definition}{Definition}[section]

\newtheorem{theorem}[definition]{Theorem}
\newtheorem{proposition}[definition]{Proposition}
\newtheorem{lemma}[definition]{Lemma}
\newtheorem{corollary}[definition]{Corollary}

\theoremstyle{remark}

\newtheorem{remark}[definition]{Remark}

\newcommand{\Hess}{\operatorname{Hess}}

\newcommand{\CB}{\mathcal{B}}
\newcommand{\CW}{\mathcal{W}}
\newcommand{\CU}{\mathcal{U}}
\newcommand{\CV}{\mathcal{V}}

\renewcommand{\div}{\operatorname{div}}
\newcommand{\rmd}{\mathrm{d}}
\newcommand{\dmu}{\,\rmd\mu}
\newcommand{\Acirc}{\accentset{\circ}{A}}
\newcommand{\Ric}{\operatorname{Ric}}
\newcommand{\IR}{\mathbf{R}}
\newcommand{\IN}{\mathbf{N}}
\newcommand{\diam}{\operatorname{diam}}
\newcommand{\dist}{\operatorname{dist}}
\newcommand{\inj}{\operatorname{inj}}

\newcommand{\Scalarcurv}[1]{\smash{\sideset{^{#1}}{}{\mathop\mathrm{Sc}\nolimits}}}

\newcommand{\ScalSig}{\!\Scalarcurv{\Sigma}}

\newcommand{\Riem}{\operatorname{Rm}}
\newcommand{\half}{\tfrac{1}{2}}

\newcommand{\Ricci}[1]{\smash{\sideset{^{#1}}{}{\mathop\mathrm{Rc}\nolimits}}}

\newcommand{\RicSig}{\!\Ricci{\Sigma}}

\newcommand{\eps}{\varepsilon}
\newcommand{\del}{\partial}
\newcommand{\qtext}[1]{\quad\text{#1}\quad}

\newcommand{\Id}{\operatorname{Id}}
\newcommand{\id}{\operatorname{id}}
\newcommand{\dd}[2]{\frac{\del #1}{\del #2}}

\newcommand{\Connection}[1]{\smash{\sideset{^{#1}}{}{\mathop\nabla\nolimits}}}

\newcommand{\Tcirc}{\accentset{\circ}{T}}
\newcommand{\tr}{\operatorname{tr}}

\newcommand{\nabM}{\!\Connection{M}}

\newcommand{\lap}{\Delta}
\newcommand{\Laplacian}[1]{\smash{\sideset{^{#1}}{}{\mathop\lap\nolimits}}}

\newcommand{\lapM}{\Laplacian{M}}

\newcommand{\al}{\alpha}
\newcommand{\be}{\beta}
\newcommand{\ga}{\gamma}

\newcommand{\ra}{\rangle}
\newcommand{\la}{\langle}
\newcommand{\Vol}{\operatorname{Vol}}

\title{Refined position estimates for surfaces of Willmore type in Riemannian manifolds}
\author{Jan Metzger\footnote{\emph{Email:} jan.metzger@uni-potsdam.de}}
\affil{University of Potsdam, Institute for
  Mathematics,\\ Karl-Liebknecht-Stra\ss{}e 24/25, 14476 Potsdam,
  Germany}
\date{}
\begin{document}
\hyphenation{}
\maketitle
\begin{abstract}%
  In this paper we consider surfaces which are critical points of the
  Willmore functional subject to constrained area.  In the case of
  small area we calculate the corrections to the intrinsic geometry
  induced by the ambient curvature. These estimates together with the
  choice of an adapted geometric center of mass lead to refined
  position estimates in relation to the scalar curvature of the
  ambient manifold.
\end{abstract}
\section{Introduction}
Let $(M,g)$ be a three dimensional Riemannian manifold and
$\Sigma\subset M$ a smooth, compact, two-sided, immersed surface. The Willmore energy of
$\Sigma$ is defined as
\begin{equation*}
  \CW(\Sigma) = \frac14 \int_\Sigma H^2 \dmu 
\end{equation*}
where $H$ is the mean curvature of the immersion and $\dmu$ denotes
the induced surface measure on $\Sigma$. We consider surfaces $\Sigma$
which are critical points of $\CW$ subject to the constraint of
prescribed area $|\Sigma|$. These surfaces satisfy the Euler-Lagrange
equation
\begin{equation}
  \label{eq:2}
  \Delta H + H |\Acirc|^2 + H \Ric(\nu,\nu) + \lambda H = 0,
\end{equation}
where $\lambda \in \IR$ is the Lagrange parameter, $\Delta$ is the
Laplace-Beltrami operator of the induced metric $\gamma$ on $\Sigma$,
$\Acirc = A -\tfrac12 H \gamma$ is the trace free part of the second
fundamental form $A$, and $\nu$ denotes (one choice of) the normal
vector to $\Sigma$. Furthermore, $\Ric$ is the Ricci curvature of
ambient metric $g$.

Concerning the existence of minimizers of the area constrained
problem in compact manifolds, we have
\begin{theorem}
  \label{thm:minimizers}
  Let $(M,g)$ be a three dimensional Riemannian manifold. Then there
  exists $a_\text{min}\in(0,\infty)$ and for each $a\in(0,a_\text{min})$ a smooth
  closed embedded surface $\Sigma^\text{min}_a$ such that
  \begin{equation*}
    \CW(\Sigma^\text{min}_a) = \inf \{ \CW(\Sigma) \mid \Sigma \text{
      smooth closed immersion and } |\Sigma| = a \}    
  \end{equation*}
  and $|\Sigma^\text{min}_a| = a$.
\end{theorem}
This was shown by Chen and Li~\cite{Chen-Li:2014} in the class of
$W^{2,2}$-conformal immersions and by Lamm and the author
\cite{lamm-metzger:2013} as well as by Mondino and Rivi\`ere
\cite{Mondino-Riviere:2013} with the additional assertion of
smoothness of the minimizing surfaces.

Critical points for this minimization problem can be constructed by
perturbing geodesic spheres centered at a non-degenerate point of the
scalar curvature. Independently Ikoma, Malchiodi, and
Mondino~\cite{Ikoma-Malchiodi-Mondino:2018} as well as Lamm, Schulze
and the author~\cite{Lamm-Metzger-Schulze:2018} have shown the
following:
\begin{theorem}
  \label{thm:foliation}
  Let $(M,g)$ be a three dimensional Riemannian manifold and let
  $p\in M$ be such that $\nabla \Scal (p) = 0$ and such that
  $\nabla^2 \Scal (p)$ is non-degenerate. Then there exist
  $a_\text{pert}\in(0,\infty)$, a neighborhood $U$ of $p$, and for each
  $a\in(0,a_\text{pert})$ a spherical surface $\Sigma^\text{pert}_a$ which
  satisfies \eqref{eq:2} for some $\lambda\in \IR$ and
  $|\Sigma_a|=a$. The $\Sigma_a$ are mutually disjoint and
  $\bigcup_{(0,a_\text{pert})} \Sigma_a = U\setminus\{p\}$.
\end{theorem}
The shape and position of critical points, that is solutions to
equation~\eqref{eq:2} was studied by Lamm and the author
~\cite{lamm-metzger:2010,lamm-metzger:2013} and with more general
assumptions by Laurin and Mondino \cite{Laurin-Mondino:2014}. The
position estimates implied by combining these three
papers are the following:
\begin{theorem}
  \label{thm:position-old}
  Let $(M,g)$ be a three dimensional Riemannian manifold. Then there
  exist $a_0\in(0,\infty)$ and constants $C_1,C_2\in(0,\infty)$ with the
  following property. Let $\Sigma\subset M$ be a surface satisfying
  equation~\eqref{eq:2} for some $\lambda\in\IR$ such that
  $|\Sigma|\leq a_0$ and $\CW(\Sigma) < 4\pi+a_0$. Then
  $\diam(\Sigma)\leq C_1|\Sigma|^{1/2}$ and
  $|\nabla\Scal| \leq C_2 |\Sigma|^{1/2}$ on $\Sigma$.
\end{theorem}
A consequence of this Theorem is that the surfaces
$\Sigma^\text{min}_a$ concentrate near critical points of the scalar
curvature of $M$. From the expansion of the Willmore functional
in~\cite[Theorem 5.1]{lamm-metzger:2010} it follows in addition that
the minimizers $\Sigma_a^\text{min}$ concentrate near the points in
$M$ where the scalar curvature is maximal as $a\to 0$.

The previously cited results were to a large extend based on the
observation, that surfaces satisfying~\eqref{eq:2} with small area and
Willmore energy close to $4\pi$ behave like their Euclidean
counterparts. The aim of this paper is to provide a more precise
description of the shape and position of solutions to~\eqref{eq:2},
that take into account the perturbations induced by the ambient
geometry.

As an application of the estimates we derive an improved position
estimate.
\begin{theorem}
  \label{thm:position-new}
  Let $(M,g)$ be a three dimensional Riemannian manifold with
  $C_B$-bounded geometry. Then there exists $a_0\in(0,\infty)$ and a
  constant $C\in(0,\infty)$ with the following property. For every
  surface $\Sigma\subset M$ satisfying equation~\eqref{eq:2} for some
  $\lambda\in\IR$ with $|\Sigma|\leq a_0$ and $\CW(\Sigma) < 4\pi+a_0$
  there exists a point $p_0$ contained in the region enclosed by
  $\Sigma$ such that for all $p\in\Sigma$ we have
  $\dist(p_0,p) < \frac{3}{4}\diam(\Sigma)$ and
  $|\nabla\Scal(p_0)| \leq C |\Sigma|$.
\end{theorem}
For the definition of $C_B$-bounded geometry, refer to
definition~\ref{def:bounded_geometry}.

The main motivation for this improvement of
Theorem~\ref{thm:position-old} is that we can use it to further narrow
down the position of the surfaces $\Sigma$ as in the theorem. To this
end assume that the critical points of the scalar curvature of $(M,g)$
are such that the Hessian there is non-degenerate, in other words,
that the scalar curvature on $M$ is a Morse function. Let $\bar p$ be
a critical point for $\Scal$ and $p$ some point near $\bar p$.  By
non-degeneracy  $|\nabla\Scal(p)|\geq c \dist(p,\bar p)$
and Theorem~\ref{thm:position-old} implies
\begin{equation*}
  c\, \dist(p,\bar p) \leq CR(\Sigma). 
\end{equation*}
From this it follows that for any neighborhood $U$ of the critical
points of $\Scal$ there is $a_0>0$ such that if $|\Sigma|<a_0$ then
$\Sigma\subset U$. However, it is not clear that a critical point of
$\Scal$ lies in the region enclosed by $\Sigma$. This is one
consequence of Theorem~\ref{thm:position-new}:
\begin{corollary}
  \label{thm:position-inside-intro}
  Let $(M,g)$ be a compact three dimensional Riemannian manifold with
  $C_B$ bounded geometry. Let
  \begin{equation*}
    Z := \{x\in M \mid \nabla\Scal(x) = 0\}
  \end{equation*}
  and assume that the Hessian $\Hess \Scal (x)$ is non-degenerate for
  every $x\in Z$.
  
  Then there exists an $a_0\in(0,\infty)$ depending only on $(M,g)$ such that for
  every surface $\Sigma$ that satisfies the Euler-Lagrange
  equation~\eqref{eq:2} for some $\lambda$, with $|\Sigma|\leq a_0$
  and $\CW(\Sigma)\leq 4\pi + a_0$ the region enclosed by $\Sigma$
  intersects $Z$ in a single point.
\end{corollary}
Note that for the $\Sigma^{\text{min}}_a$ it is automatic that
$\CW(\Sigma) \leq 4\pi + O(a)$ by comparison with geodesic
spheres. Hence, Theorem~\ref{thm:position-new} applies in particular
to these surfaces and we can be more precise:
\begin{corollary}
  \label{thm:position-inside-minimizers}
  Let $(M,g)$ be a compact three dimensional Riemannian manifold. Let
  \begin{equation*}
    Z^\text{max}:= \{x\in M \mid \Scal(x) = \max_M \Scal\}
  \end{equation*}
  and assume that the Hessian $\Hess \Scal (x)$ is non-degenerate for
  every $x\in Z^\text{max}$.
  
  Then there exists $a_0\in(0,\infty)$ with the following
  property. For every $a\in(0,a_0)$ the surface $\Sigma^\text{min}_a$ from
  Theorem~\ref{thm:minimizers} is such that it encloses a region
  that intersects $Z^\text{max}$ in a single point. 
\end{corollary}

The paper is organized as follows. In section~\ref{sec:preliminaries}
we collect some estimates from the literature and combine them to an
$L^\infty$-estimate for $\Acirc$. In section~\ref{sec:geom-cent-mass}
we introduce a geometric center of mass for small surfaces
$\Sigma\subset M$. We use this to select the point $p_0$ in the
position estimate Theorem~\ref{thm:position-new}. In
sections~\ref{sec:calculations} and~\ref{sec:estimates} we compute the
top order contributions in the expansion of certain geometric
quantities on a solution $\Sigma$ of
\eqref{eq:2}. Section~\ref{sec:expans-metr-deta} provides a
calculation of geometric identities necessary in
section~\ref{sec:position}. Theorem~\ref{thm:position-new} follows
from a slightly more precise version, Theorem~\ref{thm:position}, is
carried out in section~\ref{sec:position}. Finally,
section~\ref{sec:inside} contains the proof of
Corollary~\ref{thm:position-inside-intro}.

\subsection*{Acknowledgments}
The author acknowledges support by the DFG in form of grant ME 3816/1-2.


%
\section{Preliminaries}
\label{sec:preliminaries}
Recall the Gauss equation relating the scalar curvature $\ScalSig$ of
$\gamma$ and the scalar curvature $\Scal$ of $g$:
\begin{equation*}
  \ScalSig = \Scal - 2\Ric(\nu,\nu) + \tfrac12 H^2 - |\Acirc|^2. 
\end{equation*}
Denote the genus of $\Sigma$ by $q(\Sigma)$. Integrating the Gauss
equation and using  Gauss-Bonnet yields that
\begin{equation}
  \label{eq:gauss}
  \CW(\Sigma) = 4\pi(1-q(\Sigma)) + \frac12 \CU(\Sigma) + \CV(\Sigma)
\end{equation}
where
\begin{equation}
  \label{eq:functionaldef}
  \CU(\Sigma) = \int_\Sigma |\Acirc|^2\dmu
  \quad\text{and}\quad
  \CV(\Sigma) = \int_\Sigma \Ric(\nu,\nu) - \tfrac12 \Scal \dmu.
\end{equation}
Equation~\eqref{eq:gauss} implies that for bounded area $|\Sigma|$ and
bounded ambient curvature $|\Ric| + |\Scal|\leq C_B$ a bound for $\CU$
is equivalent to bounding $\CW$, regardless of the topology of $\Sigma$:
\begin{equation}
  \label{eq:26}
  \CW(\Sigma) \leq 4\pi + \tfrac12\CU(\Sigma) + C_B |\Sigma|.
\end{equation}
A similar bound holds in the other direction for surfaces $\Sigma$
with bounded genus $q(\Sigma)\leq q_0$:
\begin{equation*}
  \CU(\Sigma) \leq \CW(\Sigma) + 4\pi(q_0 - 1)+ C_B|\Sigma|. 
\end{equation*}
For the rest of the paper we will use these estimates for spherical
surfaces. As a consequence, an a priori bound on $\CW$ (or $\CU$) and
on $|\Sigma|$ will yields an a priori bound for the $L^2$-norm of the
second fundamental form $\|A\|^2_{L^2(\Sigma)} = \CU(\Sigma) + 2\CW(\Sigma)$.

\subsection{A priori estimates for small surfaces of Willmore type}
\label{sec:apriori-quoted}
Here we quote some estimates from the
papers~\cite{lamm-metzger:2010,lamm-metzger:2013}. They require
uniform bounds on the geometry of $(M,g)$ in the following sense.
\begin{definition}[cf.~Definition 2.1 in~\cite{lamm-metzger:2013}]
  \label{def:bounded_geometry}
  Let $(M,g)$ be a complete Riemannian manifold and
  $C_B\in(0,\infty)$. We say that $(M,g)$ has $C_B$-bounded geometry
  if for every $p\in M$ we have $\inj(M,g,p) \geq C_B^{-1}$ and
  $|\Riem(p)|+|\nabla\Riem(p)|\leq C_B$.
\end{definition}
To phrase the estimates quoted below in a geometric way, we use the
area radius of a surface defined as
\begin{equation*}
  R(\Sigma) = \sqrt{\frac{|\Sigma|}{4\pi}}.
\end{equation*}
The following lemma that if we assume small enough area and bounded
Willmore energy, then the area radius of a surface is comparable to
its diameter.
\begin{lemma}
  \label{thm:diam_comparison}
  Let $(M,g)$ be of $C_B$-bounded geometry as in
  Definition~\ref{def:bounded_geometry} and $E_0\in(0,\infty)$. Then
  there exist constants $a_0\in(0,\infty)$ and $C$ depending only on
  $C_B$ and $E_0$ with the following property: If $\Sigma\subset M$ is
  a smooth closed hypersurface with $\CW(\Sigma) \leq E_0$ and
  $|\Sigma|\leq a_0$ then 
  \begin{equation*}
    C^{-1} \diam(\Sigma) \leq R(\Sigma) \leq C \diam(\Sigma).
  \end{equation*}
\end{lemma}
\begin{proof}
  The right estimate follows from Lemma~2.2 in
  \cite{lamm-metzger:2010} whereas the left estimate is a direct
  consequence of Lemma~2.5 in \cite{lamm-metzger:2013}.
\end{proof}
\begin{theorem}[cf.~Theorem~5.4 from \cite{lamm-metzger:2013}]
  \label{thm:apriori}
  Let $(M,g)$ be of $C_B$-bounded geometry as in
  Definition~\ref{def:bounded_geometry}. Then there exist constants
  $a_0\in(0,\infty)$ and $C$ depending only on $C_B$ such that for
  every surface $\Sigma$ satisfying~\eqref{eq:2}, with $|\Sigma|\leq
  a_0$ and $\CW(\Sigma)\leq 4\pi+ a_0$ we have the estimate
  \begin{equation*}
    \int_\Sigma |\nabla^2 H|^2 + H^2|\nabla A|^2 +
    H^4|\Acirc|^2 \dmu
    \leq
    C.
  \end{equation*}
\end{theorem}
\begin{lemma}[cf.~Corollary 5.5 from \cite{lamm-metzger:2013}]
  \label{thm:Hlinfty}
  Under the assumptions of theorem~\ref{thm:apriori} we have that
  \begin{equation*}
    \|\Acirc\|_{L^2(\Sigma)} \leq CR(\Sigma)^2
    \quad\text{and}\quad
    \left\| H - \tfrac{2}{R(\Sigma)}\right\|_{L^\infty(\Sigma)} \leq CR(\Sigma).
  \end{equation*}
\end{lemma}
This lemma shows that for small enough area $|\Sigma|$ the mean
curvature $H>0$ and thus the estimates from \cite{lamm-metzger:2010}
hold under the assumptions of Theorem~\ref{thm:apriori}. It also
follows that $\|A\|_{L^2(\Sigma)} \leq C$.

For $r\in(0,\infty)$, denote the intrinsic ball centered at $p\in M$
by 
\begin{equation*}
  \CB_r(p)=\{q\in M \mid \dist_g(p,q)\leq r\}.
\end{equation*}
\begin{proposition}[cf.~Corollary 3.6 from \cite{lamm-metzger:2010}]
  \label{thm:lambda_estimate}
  Let $(M,g)$ and $\Sigma$ be as in Theorem~\ref{thm:apriori}. Let
  $p_0\in M$ be a point with $\Sigma \subset \CB_{2\diam(\Sigma)}(p_0)$, then
  \begin{equation*}
    \left|\lambda + \tfrac13 \Scal(p_0)\right| \leq CR(\Sigma).
  \end{equation*}
\end{proposition}
\subsection{General inequalities}
\label{sec:general-inequalities}
The Bochner identity for surfaces can be stated as follows.
\begin{lemma}
  \label{thm:bochner}
  For all functions $f\in C^\infty(\Sigma)$ we have that
  \begin{equation*}
    \int_\Sigma 2 |\nabla^2 f|^2 + \half H^2 |\nabla f|^2 \dmu
    \leq
    \int_\Sigma 2|\Delta f|^2 + \big(2\Ric(\nu,\nu) - \Scal +
    |\Acirc|^2\big) |\nabla f|^2 \dmu.
  \end{equation*}
\end{lemma}
\begin{proof}
  The Bochner identity states that
  \begin{equation*}
    \int_\Sigma |\nabla^2 f|^2 \dmu = \int_\Sigma (\Delta f)^2  -
    \RicSig(\nabla f, \nabla f)\dmu.
  \end{equation*}
  Since $\Sigma$ is a surface, its Ricci curvature satisfies $\RicSig
  = \frac{1}{2} \ScalSig \gamma$ and the scalar curvature $\ScalSig$
  of $\Sigma$ can be expressed via the the Gauss equation
  \begin{equation*}
    \ScalSig = \Scal- \Ric(\nu,\nu) +\tfrac12 H^2 - |\Acirc|^2.
  \end{equation*}
\end{proof}
Let $C$ and $a_0$ be the constants from
lemma~\ref{thm:diam_comparison}. Hence, if $\Sigma$ is such that
$|\Sigma|\leq a_0$ and $p\in M$ is some point such that there exists
$x\in\Sigma$ with $\dist(p,x) \leq R(\Sigma)$, then
$\Sigma \subset \CB_{(C+1)R(\Sigma)}(p)$. Hence, there exists
$a_0'\in(0,a_0)$ such that if $|\Sigma|\leq a_0'$ and
$\rho =(C+1)R(\Sigma)$ then $\rho \leq \inj(M,g,p)$.  In this case
there are normal coordinates
$x:\CB_\rho(p) \to B_\rho(0)\subset\IR^3$. The metric $(x^{-1})^*g$ on
$B_\rho(0)$ has the expansion $(x^{-1})^*g = \delta + h$ with
\begin{equation*}
  \sup\big( |x|^{-2}|h| + |x|^{-1} |\del h| + |\del^2 h|\big) \leq
  h_0.
\end{equation*}
Here $h_0$ is a constant depending only on $C_B$ and $\del$ denotes
partial derivatives in the coordinate system given by $x$. In
particular, we can apply the following two estimates on surfaces
$\Sigma$ as in section~\ref{sec:apriori-quoted} with uniform
constants, that is constants independent of $\Sigma$, provided
$|\Sigma|\leq a_0$ for some constant $a_3\in(0,\infty)$ depending only
on $C_B$.
\begin{lemma}
  \label{thm:michael-simon-sobolev}
  Let $g=g^E+h$ on $B_\rho$ and $C_0$ be given. Then there exists
  $\rho_0\in(0,\rho)$ and a constant $C$ depending only on $\rho$,
  $h_0$ and $C_0$ such that for all surfaces
  $\Sigma\subset B_{\rho_0}$ with $\|A\|_{L^2(\Sigma)}\leq C_0$ and
  all $f\in C^\infty(\Sigma)$ we have
  \begin{equation*}
    \left(\int_\Sigma f^2 \dmu \right)^{1/2}
    \leq
    C \int_\Sigma |\nabla f| + |Hf| \dmu.
  \end{equation*}
\end{lemma}
In addition, for surfaces as in Theorem~\ref{thm:apriori} we have a
Poincar\'e inequality of the following form:
\begin{lemma}
  \label{thm:poincare}
  Let $(M,g)$ be of $C_B$-bounded geometry as in
  Definition~\ref{def:bounded_geometry}. Then there exist constants
  $a_0\in(0,\infty)$, $\eps\in(0,\infty)$ and $C$ depending only on
  $C_B$ with the following property. If $\Sigma\subset M$ is a smooth
  closed hypersurface with $\CU(\Sigma) \leq \eps$ and $|\Sigma|\leq
  a_0$ then for all $f\in C^\infty(\Sigma)$ we have
  \begin{equation*}
    \int_\Sigma |f -\bar f|^2 \dmu
    \leq
    C|\Sigma| \int_\Sigma |\nabla f|^2 \dmu.
  \end{equation*}
  Here
  \begin{equation*}
    \bar f = |\Sigma|^{-1} \int_\Sigma f \dmu
  \end{equation*}
  is the mean value of $f$.
\end{lemma}
\begin{proof}
  This holds without assuming an upper bound for the area of $\Sigma$
  if $(M,g)$ is Euclidean space in view of the eigenvalue estimates of
  DeLellis and M\"uller \cite[Corollary 1.3]{delellis-muller:2006} for
  nearly umbilical surfaces.

  For general $(M,g)$ we first use inequality~\eqref{eq:26} to bound
  $\CW(\Sigma)$ in terms of $\CU(\Sigma)$ and $|\Sigma|$. Then by
  Lemma~\ref{thm:diam_comparison} we infer that small area $|\Sigma|$
  implies small diameter. Using normal coordinates covering
  $\Sigma$ implies that we are in a nearly Euclidean setting. It is
  then straight forward to deduce the desired Poincar\'e inequality on
  $\Sigma$ with respect to the metric induced by $g$ from the one with
  respect to the Euclidean metric in the normal coordinate system.
\end{proof}
The following estimate follows from the Michael-Simon-Sobolev
inequality from Lemma~\ref{thm:michael-simon-sobolev} and can be
proved exactly as \cite[Lemma 2.8]{kuwert-schatzle:2001}. This form
appears in \cite[Lemma 3.7]{lamm-metzger:2010}.
\begin{lemma}
  \label{thm:linfty-embedding}
  Assume that the metric $g=g^E +h$ on $B_\rho$ is given. Then there
  exist $\rho_0\in(0,\rho)$ and a constant $C<\infty$ such that for
  all surfaces $\Sigma \subset B_r$ with $r\in(0,\rho_0)$ and for all
  smooth forms $\phi$ on $\Sigma$ we have
  \begin{equation*}
    \|\phi\|_{L^\infty(\Sigma)}^4
    \leq
    C \|\phi\|^2_{L^2(\Sigma)} \int_\Sigma |\nabla^2\phi|^2 + |H|^4 |\phi|^2 \dmu.
  \end{equation*}
\end{lemma}
\subsection{A $L^\infty$ estimate for $\protect\Acirc$}
\label{sec:apriori-new}
For the later exposition we need two more estimates not
present in \cite{lamm-metzger:2010,lamm-metzger:2013}.
\begin{lemma}
  \label{thm:acirc_linfty}
  Let $(M,g)$ be as in definition~\ref{def:bounded_geometry}. Then
  there exist constants $a_0\in(0,\infty)$ and $C\in(0,\infty)$
  depending only on $C_B$ such that if $\Sigma$ satisfies~\eqref{eq:2}
  for some $\lambda\in\IR$ with $|\Sigma|\leq a_0$ and
  $\CW(\Sigma)\leq 4\pi + a_0$ then 
  \begin{equation*}
    \|\Acirc\|_{L^\infty(\Sigma)} \leq CR(\Sigma)
  \end{equation*}
  and
  \begin{equation*}
    \|H^{-1} -R(\Sigma)/2\|_{L^\infty(\Sigma)} \leq CR(\Sigma)^2.
  \end{equation*}
\end{lemma}
\begin{proof}
  We assume that $a_0$ is so small that the estimates from
  theorem~\ref{thm:apriori} apply. Then the second estimate is an
  immediate consequence from the second estimate
  in~\ref{thm:Hlinfty}. To show the first estimate, proceed as in the
  proof of~\cite[Lemma 15]{lamm-metzger-schulze:2009}. In view of the
  Bochner identity it suffices to estimate $\Delta \Acirc$ in the
  $L^2$-norm. The $L^\infty$ estimate then follows from
  lemma~\ref{thm:linfty-embedding}. To derive this estimate, recall
  the Simons identity \cite{simons:1968} in the form of
  \cite[eq.~(8)]{lamm-metzger-schulze:2009}
  \begin{equation}
    \label{eq:24}
    \begin{split}
      \Delta \Acirc_{ij} &= (\nabla^2 H)^\circ_{ij} +
      H\Acirc_i^k\Acirc_{kj} + \tfrac12H^2\Acirc_{ij}
      -|\Acirc|^2\Acirc_{ij} -\tfrac12 H|\Acirc|^2\gamma_{ij}
      \\
      &\phantom{=}
      + \Acirc_j^k \gamma_{lm} \Riem_{likm} + \Acirc^{kl}\Riem_{ikjl}
      + 2 \nabla_i \omega_j - \div\omega \gamma_{ij}.
    \end{split}
  \end{equation}
  Here $\omega = \Ric(\nu,\cdot)^T$ denotes the tangential 1-form
  obtained from projecting the 1-from $\Ric(\nu,\cdot)$ to the tangent
  space of $\Sigma$. From the calculation in
  section~\ref{sec:covar-deriv-omega} we get that
  \begin{equation*}
    |\nabla\omega| \leq |\nabla\Ric| + |A| |\Ric|.
  \end{equation*}
  This yields
  \begin{equation*}
    \|\nabla\omega\|_{L^2(\Sigma)}^2
    \leq C |\Sigma| + \int_\Sigma |A|^2 \dmu \leq C.
  \end{equation*}
  Together with equation~\eqref{eq:24} this gives
  \begin{equation*}
    \begin{split}
      \|\Delta \Acirc\|_{L^2(\Sigma)}
      &\leq
      c\big(\|\nabla^2 H\|_{L^2(\Sigma)}  + \|H^2\Acirc\|_{L^2(\Sigma)}
      + \|\Acirc\|_{L^6}^2 \\
      &\phantom{\leq c\big(}
      + \|A\|_{L^2}\|\Riem\|_{L^\infty(\Sigma)} + \|\nabla\Ric\|_{L^2(\Sigma)} \big).
      \\
      &\leq C + c \|\Acirc\|^2_{L^\infty(\Sigma)}\|\Acirc\|_{L^2(\Sigma)}.
    \end{split}
  \end{equation*}
  Here $c$ denotes a purely numerical constant, and we used the
  estimates from section~\ref{sec:apriori-quoted} in the second step.

  From the Bochner identity~\ref{thm:bochner} (more precisely a variant for
  two tensors) we obtain that
  \begin{equation*}
    \begin{split}
      &\|\nabla^2 \Acirc\|_{L^2(\Sigma)} + \|H \nabla \Acirc\|_{L^2(\Sigma)}
      \\
      &\quad\leq
      c \|\Delta \Acirc\|_{L^2(\Sigma)} +
      C(1+\|\Acirc\|_{L^\infty(\Sigma)}) \|\nabla \Acirc\|_{L^2(\Sigma)}
      \\
      &\quad\leq
      C +  c \|\Acirc\|^2_{L^\infty(\Sigma)}\|\Acirc\|_{L^2(\Sigma)}
      +  C(1+\|\Acirc\|_{L^\infty(\Sigma)}) \|\nabla \Acirc\|_{L^2(\Sigma)}
    \end{split}
  \end{equation*}
  Note that the last term on the right hand side can be absorbed to
  the left, if $R$ is small enough.

  This yields in view of lemma~\ref{thm:linfty-embedding} that
  \begin{equation*}
    \begin{split}
      \|\Acirc\|_{L^\infty(\Sigma)}^4
      &\leq
      C \|\Acirc\|_{L^2(\Sigma)}^2 \big( \|\nabla^2\Acirc\|_{L^2(\Sigma)}^2 +
      \|H^2\Acirc\|_{L^2(\Sigma)} \big)
      \\
      &\leq
      C R(\Sigma)^4 ( C + C \|\Acirc\|^4_{L^\infty(\Sigma)} R(\Sigma)^2).
    \end{split}
  \end{equation*}
  If $R(\Sigma)$ is small enough, this gives
  \begin{equation*}
    \|\Acirc\|_{L^\infty(\Sigma)} \leq CR(\Sigma).
  \end{equation*}
  Note that by choosing $a_0$ small enough, we can ensure that
  $R(\Sigma)$ is so small, that the above steps apply to $\Sigma$ as
  in the assumption.
\end{proof}

\subsection{Approximately spherical surfaces}
In this section we discuss the approximation of a given surface
$\Sigma\subset\IR^3$ by spheres. The main tool here are the estimates
from DeLellis and
M\"uller~\cite{delellis-muller:2005,delellis-muller:2006}. We quote
their estimates in the form needed here from \cite[Theorem
2.4]{lamm-metzger:2010}. These results are purely Euclidean. To
distinguish geometric quantities computed with respect to the
Euclidean metric we use the superscript~$^E$.
\begin{theorem}
  \label{thm:delellis-muller}
  There exists a universal constant $C$ with the following properties.
  Assume that $\Sigma\subset\IR^3$ is a surface with
  $\|\Acirc^E\|^2_{L^2(\Sigma,\gamma^E)}< 8\pi$.  Let $R^E :=
  \sqrt{|\Sigma|^E/4\pi}$ be the Euclidean area radius of $\Sigma$ and
  $a^E := |\Sigma|_E^{-1} \int_\Sigma x \dmu^E$ be the Euclidean
  center of gravity. Then there exists a conformal map $F:
  S:=S_{R^E}(a^E) \to \Sigma\subset\IR^3$ with the following
  properties. Let $\gamma^S$ be the standard metric on $S$, $N$ the
  Euclidean normal vector field and $\phi$ the conformal factor, that is
  $F^*\gamma^E = \phi^2 \gamma^S$. Then the following estimates hold
  \begin{align*}
    \| H^E - 2/R^E \|_{L^2(\Sigma,\gamma^E)}
    &\leq
    C \|\Acirc^E\|_{L^2(\Sigma,\gamma^E)}
    \\
    \|F - \id_S \|_{L^\infty(S)}
    &\leq
    C R^E \|\Acirc^E\|_{L^2(\Sigma,\gamma^E)}
    \\
    \|\phi^2 -1\|_{L^\infty(S)}
    &\leq
    C  \|\Acirc^E\|_{L^2(\Sigma,\gamma^E)}
    \\
    \|N - \nu^E\circ F\|_{L^2(S)}
    &\leq
    C R^E \|\Acirc^E\|_{L^2(\Sigma,\gamma^E)}.
  \end{align*}
\end{theorem}
These estimates can be applied in our situation by choosing
appropriate normal coordinates near a small surfaces as described in
section~\ref{sec:general-inequalities} and compare geometric
quantities in the given metric to the Euclidean background. In
particular we have:
\begin{lemma}[cf. \protect{\cite[Lemma 2.5]{lamm-metzger:2010}}]
  \label{thm:euclidean-umbilic}
  Let $g=g^E+h$ on $B_\rho$ be given. Then there exists $0<\rho_0 <\rho$
  and a constant $C$ depending only on $\rho$ and $h_0$ such that for all surfaces $\Sigma\subset
  B_r$ with $r<\rho_0$ we have
  \begin{equation*}
    \begin{split}
      \|\Acirc^E\|^2_{L^2(\Sigma,\gamma^E)}
      &\leq
      C \|\Acirc\|^2_{L^2(\Sigma,\gamma)} + C\rho^4 \|H\|^2_{L^2(\Sigma,\gamma)}
    \end{split}
  \end{equation*}
\end{lemma}


%
\section{A geometric center of mass}
\label{sec:geom-cent-mass}
The calculations in section~\ref{sec:position} require that the normal
coordinates in which we look at our surfaces $\Sigma$ are well adapted
to $\Sigma$. In this section we propose one way to assign a geometric
center of mass to our surfaces. Centering the normal coordinates there
gives good control on the center of mass of the image of the surface
in the coordinate picture.

Let $(M,g)$ be a Riemannian manifold and $\Sigma\subset M$ a closed
smooth hypersurface with extrinsic diameter $d = \diam(\Sigma) = \max
\{\dist(x,y) \mid x,y\in \Sigma\}$ where $\dist$ denotes the
distance function in $(M,g)$. Assume that $2d<\inj(M,g)$. For
$p\in M$ let $d_p(x) :=\dist(p,x)$ and set
\begin{equation*}
  w(p) := \int_\Sigma d_p(x)^2 \dmu.
\end{equation*}
Then $w$ is a smooth, positive, proper function on $M$ which attains
its global infimum on the compact set
\begin{equation*}
  K := \{ p\in M\mid \dist(p,\Sigma) \leq d\}.
\end{equation*}
This follows from comparing values of $w$ outside of $K$ with $w(p)$
for some $p\in\Sigma$. Let $p_0\in K$ be a point where $w$ attains its
minimum. Since $p_0\in K$ we have that $\Sigma \subset \CB_{2d}(p_0)$
and since $2d < \inj(M,g)$ we find that $\Sigma$ is completely
contained in a normal coordinate neighborhood centered at $p_0$. Let
$\psi:\CB_\rho(p_0) \to B_\rho(0) \subset \IR^n$ be such normal
coordinates where $\rho>2d$ denotes the injectivity radius on $(M,g)$ at
$p_0$. Let $x\in B_\rho(0)$ and $p = \psi^{-1}(x)$. Then
\begin{equation*}
  \tilde w(x) := w(p) = \int_{\psi(\Sigma)} \dist_g(x,y)^2 \dmu_g (y)
\end{equation*}
where $\dist_g$ is the distance function induced by the pull-back
metric $(\psi^{-1})^*g$ to $B_\rho(0)$ and $\dmu_g$ denotes the induced
surface measure.

Since $w$ is critical at $p_0$ also $\tilde w$ is critical at $0$ and
we compute
\begin{equation*}
  0 = \dd{}{x^\alpha} \tilde w(0) = 2\int_{\psi(\Sigma)} y^\alpha \dmu_g
\end{equation*}
since in normal coordinates $\dist_g(x,y)^2 = |x-y|^2 + O(|x|^2)$. 

We can also change the surface measure to the Euclidean one, recording
the error term:
\begin{equation*}
  \left| \int_{\psi(\Sigma)} y^\alpha \dmu_g - \int_{\psi(\Sigma)}
    y^\alpha \dmu^E \right| \leq
  C d^3 |\Sigma|.
\end{equation*}
Here and in the following we use $y$ to refer to the position vector on
$\psi(\Sigma)$.

Summarizing, we arrive at the following:
\begin{lemma}
  \label{thm:geometric-center}
  Let $(M,g)$ be of $C_B$-bounded geometry. Then there exists a constant $C$
  depending only on $C_B$ with the following property: For every
  closed smooth hypersurface $\Sigma\subset M$ with extrinsic diameter
  $d = \diam(\Sigma) = \max \{\dist(x,y) \mid x,y\in \Sigma\} <
  \frac12 \inj(M,g)$ there exists a point $p_0\in M$ with $\dist(p_0,\Sigma)\leq d$
  such that in normal coordinates centered at $p_0$ we have that
  \begin{equation*}
    \int_{\psi(\Sigma)} y^\alpha \dmu_g = 0
    \qtext{and}
    \left|\int_{\psi(\Sigma)} y^\alpha \dmu^E\right| \leq C d^3 |\Sigma|.
  \end{equation*}
\end{lemma}
Combined with theorem~\ref{thm:delellis-muller} and
lemma~\ref{thm:euclidean-umbilic} we obtain the following estimate in
the case where $\Sigma$ is a surface in a $3$-dimensional manifold:
\begin{lemma}
  \label{thm:sphere-approximation} 
  Let $(M,g)$ be three dimensional and of $C_B$-bounded geometry. Then
  there exist constants $C$ and $a_0\in(0,\infty)$ depending only on
  $C_B$ with the following property: For every closed smooth surface
  $\Sigma\subset M$ with $|\Sigma|\leq a_0$ and $\CU(\Sigma) \leq a_0$
  there exists a point $p_0\in M$, normal coordinates
  $\psi:\CB_\rho(p_0) \to B_\rho(0)\subset\IR^3$ and in these
  coordinates we have that
  \begin{equation}
    \label{eq:10}
    \| \tfrac{y}{R} - \nu \|_{L^2(\Sigma)} \leq C \big(R^3 + R\|\Acirc\|_{L^2(\Sigma)}\big)
  \end{equation}
  and
  \begin{equation}
    \label{eq:29}
    \| \dist (p_0,\cdot) - R \|_{L^\infty(\Sigma)}
    \leq
    C\big(R^3 + R\|\Acirc\|_{L^2(\Sigma)}\big).
  \end{equation}  
  Here $R= R(\Sigma)$ denotes the area radius of $\Sigma$.
\end{lemma}
\begin{proof}
  We choose $a_0 \in (0,1]$ in a moment. By~\eqref{eq:26} this gives
  the a priori bound $\CW(\Sigma)\leq 4\pi + \half + C_B$. In view of
  the diameter bound from Lemma~\ref{thm:diam_comparison} we can
  choose $a_0\in(0,1]$ so small that Lemma~\ref{thm:geometric-center}
  holds for $\Sigma$ as in the assumption. Let $\psi:\CB_\rho(p_0) \to
  B_\rho(0)$ denote the coordinates from there. In view of
  the diameter estimate the quantities $d=\diam(\Sigma)$ from
  Lemma~\ref{thm:geometric-center} and $R$ are comparable. Hence also
  $\max_{p\in\Sigma}\dist(p,p_0)\leq CR$ so that the
  estimate from Lemma~\ref{thm:euclidean-umbilic} can be rephrased as
  \begin{equation}
    \label{eq:25}
    \|\Acirc^E\|^2_{L^2(\Sigma,\gamma^E)} \leq  C(\CU(\Sigma) + R^4)
  \end{equation}
  where we also used that the Willmore functional is a priori bounded.

  To prove~\eqref{eq:10}, it is thus sufficient to prove the Euclidean
  inequality
  \begin{equation*}
    \| \tfrac{y}{R} - \nu^E \|_{L^2(\Sigma,\gamma^E)} \leq C \big(R^3 + R\|\Acirc^E\|_{L^2(\Sigma,\gamma^E)}\big)
  \end{equation*}
  since in the previous coordinates we have that $\nu-\nu^E = O(R^2)$
  and due to equation~\eqref{eq:25}. Compute
  \begin{equation*}
    \nabla^E \tfrac{y}{R} = R^{-1}\Id
    \qtext{and}
    \nabla^E \nu^E = \tfrac{H^E}{2} \Id + \Acirc^E.
  \end{equation*}
  Here we denote the tangential derivative along $\Sigma$ by
  $\nabla^E$, $\Id$ denotes the identity endomorphism field in the
  tangent bundle on $\Sigma$ and we slightly abuse notation by not
  distinguishing $\Acirc^E$ from its associated endomorphism. This
  gives the estimate
  \begin{equation*}
    \begin{aligned}
      \|\nabla \left(\tfrac{y}{R} -\nu^E\right) \|_{L^2(\Sigma,\gamma^E)}
      &\leq
      \tfrac{1}{2} \|\tfrac{2}{R} - H^E\|_{L^2(\Sigma,\gamma^E)} +
      \|\Acirc^E\|_{L^2(\Sigma,\gamma^E)}
      \\
      &\leq
      C \|\Acirc^E\|_{L^2(\Sigma,\gamma^E)}.
    \end{aligned}
  \end{equation*}
  The last inequality follows from Theorem~\ref{thm:delellis-muller}
  if $a_0\in(0,1]$ is chosen so small that equation~\eqref{eq:25} implies
  $\|\Acirc^E\|^2_{L^2(\Sigma,\gamma^E)} \leq 6\pi$.

  By choosing $a_0\in(0,1]$ even smaller, we can ensure that the Poincar\'e
  inequality from  Theorem~\ref{thm:poincare} holds. This gives
  \begin{equation*}
    \|\tfrac{y}{R} -\nu^E - m \|_{L^2(\Sigma,\gamma^E)} \leq C R \|\Acirc^E\|_{L^2(\Sigma,\gamma^E)}.
  \end{equation*}
  where 
  \begin{equation*}
    m = |\Sigma|^{-1} \int_{\psi(\Sigma)} (\tfrac{y}{R} -\nu^E)\dmu^E
    = |\Sigma|^{-1} \int_{\psi(\Sigma)} \tfrac{y}{R} \dmu^E.
  \end{equation*}
  By Lemma~\ref{thm:geometric-center} we have $|m| \leq CR^2$ that is
  $\|m\|_{L^2(\Sigma)} \leq CR^3$ and thus the first of the claimed
  estimate follows.

  To show equation~\eqref{eq:29} observe that the Euclidean center of
  gravity $a^\Sigma$ of $\psi(\Sigma)$ satisfies
  \begin{equation*}
    |a^E| = |\Sigma|_E^{-1} \left| \int_{\psi(\Sigma)} y \dmu^E
    \right| \leq CR^3. 
  \end{equation*}
  Parameterizing $\psi(\Sigma)$ with a map
  $F:S_{R^E}(a^E) \to \psi(\Sigma)$ as in
  Theorem~\ref{thm:delellis-muller}, we get that for every
  $y\in S_{R^E}(a^E)$
  \begin{equation*}
    |y-a^E| - |a^E| - |F(y) -y| \leq |F(y)| \leq |y-a^E| + |a^E| + |F(y) -y|
  \end{equation*}
  so that
  \begin{equation*}
    \big||F(y)| - R^E\big|
    \leq |a^E| + |F(y) -y|
    \leq CR^3 + CR^E \|\Acirc^E\|_{L^2(\Sigma,\gamma^E)}.    
  \end{equation*}
  Since we are in normal coordinates around $p_0$ we have that for all
  $x\in\Sigma$ that $\dist(x,p_0) = |\psi(x)| = |F(F^{-1}(\psi(x)))|$
  and the second claim follows.
\end{proof}
  
\begin{corollary}
  \label{thm:even-odd}
  Let $(M,g)$ and $\Sigma$ satisfy the assumptions of
  Lemma~\ref{thm:sphere-approximation} and let $\psi$ be as
  there. Then for every $k\in\IN\cup\{0\}$ there is a constant $C_k$
  depending only on the constant $C$ in
  Lemma~\ref{thm:sphere-approximation} and on $k$ such that
  \begin{equation*}
    \left|\int_{\psi(\Sigma)} \prod_{l=1}^{2k+1} \rho_l \dmu\right|
    \leq C_k \big(R^3 + R\|\Acirc\|_{L^2(\Sigma)}\big).
  \end{equation*}
  Here, for every $l\in\{1,\ldots,2k+1\}$ we can choose $\rho_l$
  freely from the functions $\{\nu^\alpha,\frac{y^\alpha}{R} \mid \alpha=1,2,3\}$.  
\end{corollary}
\begin{proof}
  Note that if $k=0$ then the claim directly follows from
  Lemma~\ref{thm:geometric-center} if $\rho_1 = \frac{y^\alpha}{R}$
  and from the fact that $\int_{\psi(\Sigma)}(\nu^E)^\alpha \dmu^E =
  0$ if $\rho_1=\nu^\alpha$ for some $\alpha=1,2,3$. For brevity, we 
  indicate the proof only in the case $k=1$ below. Also note that it is
  sufficient to consider the Euclidean setting, that is with
  $\rho_l \in \{(\nu^E)^\alpha,\frac{y^\alpha}{R} \mid \alpha=1,2,3\}$
  and with $\|\Acirc\|_{L^2(\Sigma)}$ replaced by
  $\|\Acirc^E\|_{L^2(\Sigma,\gamma^E)}$ using the same reduction as in
  the proof of Lemma~\ref{thm:sphere-approximation}. 

  We proceed in two steps: in the first step, we use
  Theorem~\ref{thm:delellis-muller} to prove the case
  $\rho_l \in \{(\nu^E)^\alpha \mid \alpha=1,2,3\}$, in the second
  step we use Lemma~\ref{thm:sphere-approximation} to conclude.

  {\itshape Step 1:} Let $\rho_l = (\nu^E)^{\alpha_l}$ for $l=1,2,3$ and
  $\alpha_l\in\{1,2,3\}$. Let $R^E$ and $a^E$ as in
  Theorem~\ref{thm:sphere-approximation} and denote
  $S:=S_{R^E}(a^E)$. Let $F:\to \Sigma$ be the parameterization from
  Theorem~\ref{thm:sphere-approximation}, $N:S \to S^2$ be the
  normal of $S$ and $\tilde\rho_l := \rho_l\circ F$. Then we can write
  \begin{equation*}
    \begin{split}
      &\int_{\psi(\Sigma)} \rho_1\rho_2\rho_3 \dmu^E 
      =
      \int_S \tilde\rho_1 \tilde\rho_2 \tilde\rho_3 \phi^2 \dmu^E 
      \\
      &\quad=
      \int_S N^{\alpha_1}N^{\alpha_2}N^{\alpha_3}
      +
      N^{\alpha_1}N^{\alpha_2}N^{\alpha_3} (\phi^2-1) +
      (\tilde\rho_1 - N^{\alpha_1})N^{\alpha_2}N^{\alpha_3} \phi^2
      \\
      &\qquad\qquad
      + \tilde\rho_1(\tilde\rho_2 - N^{\alpha_2}) N^{\alpha_3}
      \phi^2  
      + \tilde\rho_1 \tilde\rho_2(\tilde\rho_3 -
      N^{\alpha_3}) \phi^2\dmu^E. 
    \end{split}
  \end{equation*}
  Since $S$ is a sphere
  $\int_S N^{\alpha_1}N^{\alpha_2}N^{\alpha_3}\dmu^E = 0$ and thus,
  using Cauchy-Schwarz in the first inequality and
  Theorem~\ref{thm:delellis-muller} in the last inequality we conclude
  \begin{equation*}
    \begin{split}
      &\left|\int_{\psi(\Sigma)} \rho_1\rho_2\rho_3 \dmu^E\right|
      \\
      &\quad\leq
      C R \|N-\nu^E\circ R\|_{L^2(S,\gamma^E)} + C R^2 \|\phi^2 -1
      \|_{L^{\infty}(S)}
      \leq
      C R^2 \|\Acirc^E\|_{L^2(\Sigma,\gamma^E)}.
    \end{split}
  \end{equation*}
  Note that this implies the claimed inequality.

  {\itshape Step 2:} Assume that $\rho_l= (\nu^E)^{\alpha_l}$ or $\rho_l =
  \frac{y^{\alpha_l}}{R}$ for $l=1,2,3$. We can use a telescope sum as
  above and Cauchy-Schwarz to estimate 
  \begin{equation*}
    \left| \int_{\psi(\Sigma)} \rho_1\rho_2\rho_3 -
      (\nu^E)^{\alpha_1}(\nu^E)^{\alpha_2}(\nu^E)^{\alpha_3} \dmu
    \right|
    \leq
    C R \sum_{l=1}^3 \| \rho_l - (\nu^E)^{\alpha_l}
    \|_{L^2(\Sigma,\gamma^E)}. 
  \end{equation*}
  Note that the terms in the sum on the right either vanish or can be
  bounded using Lemma~\ref{thm:sphere-approximation}. We thus arrive
  at the claimed inequality.
\end{proof}
Note that products of an even number of factors can be treated in a
similar fashion as above and equal the respective integrals on a
centered round sphere up to the same error term as above.

%
\section{Geometric identities}
\label{sec:calculations}
Throughout this section we assume that $(M,g)$ has $C_B$-bounded
geometry and that $\Sigma\subset M$ is a closed, immersed, smooth
surface such that 
\begin{enumerate}
\item $\Sigma$ satisfies equation~\eqref{eq:2}.
\item $|\Sigma| \leq a_0$ and $\CW(\Sigma)\leq 4\pi + a_0$.
\end{enumerate}
Here we assume that $a_0$ is so small that the estimates from
Theorem~\ref{thm:apriori} and Lemmas~\ref{thm:Hlinfty}
and~\ref{thm:acirc_linfty} hold.

To shorten the exposition, we augment the big-$O$ notation as
follows. If $f$ is some quantity defined on a surface $\Sigma$ as
above, we say $f = O_{L^p}(R^k)$  if 
\begin{equation*}
  \int_\Sigma f^p \dmu \leq C R^{pk+2},
\end{equation*}
where $R=R(\Sigma)$ refers to the area radius of $\Sigma$. We also use
this for $p=\infty$, that is $f = O_{L^\infty}(R^k)$ denotes
\begin{equation*}
  \|f\|_{L^\infty(\Sigma)} \leq CR^k.
\end{equation*}
Using this notation, the a priori estimates from section~\ref{sec:preliminaries} can be
stated as follows:
\begin{equation*}
  \Acirc = O_{L^\infty}(R),
  \qquad
  \nabla A = O_{L^2}(1),
  \quad\text{and}\quad
  \nabla^2 H = O_{L^2}(R^{-1}).
\end{equation*}
Lemmas~\ref{thm:Hlinfty} and~\ref{thm:acirc_linfty}
imply that $H=O_{L^\infty}(R^{-1})$ and $H^{-1}=O_{L^\infty}(R)$.

The following computations are done in abstract index notation, where
Latin indices $i,j,k,\ldots\in \{1,2\}$ refer to an local orthonormal
frame $\{e_1,e_2\}$ on $\Sigma$ and $\nu$ denotes a choice of normal
to $\Sigma\subset M$ so that
$A_{ij}:= A(e_i,e_j)=g(\nabM_{e_i} \nu,e_j)$ where $\nabM$ denotes the
Levi-Civita connection on $(M,g)$. Tangential derivatives to $\Sigma$
are denoted by $\nabla$.

\subsection{Hessian of $\Ric(\nu,\nu)$}
We begin by calculating the gradient
\begin{equation}
  \label{eq:8}
  \begin{split}
    \nabla_i \Ric(\nu,\nu)
    &=
    (\nabM_{e_i} \Ric)(\nu,\nu) + 2 A_{ik} \Ric(e_k,\nu)
    \\
    &=
    (\nabM_{e_i}\Ric)(\nu,\nu)  + H \Ric(e_i,\nu) + 2
    \Acirc_{ik}\Ric(e_k,\nu)
    \\
    &=
    H \omega_i + O_{L^\infty}(1).
  \end{split}
\end{equation}
As before $\omega = \Ric(\nu,\cdot)^T$ denotes the tangential
projection of the 1-form $\Ric(\nu,\cdot)$ to $\Sigma$. In the second
step we used the splitting
\begin{equation*}
  A_{ij} = \Acirc_{ij} + \tfrac12 H \gamma_{ij}.
\end{equation*}
Differentiating further yields
\begin{equation}
  \label{eq:5}
  \begin{split}
    \nabla_{i,j}^2 \Ric(\nu,\nu)
    &=
    (\nabM^2_{i,j} \Ric)(\nu,\nu)
    -(\nabM_\nu \Ric)(\nu,\nu) A_{ij}
    \\
    &\quad
    + 2 (\nabM_{e_i}\Ric)(e_k,\nu)A_{jk}
    + 2 (\nabM_{e_j}\Ric)(e_k,\nu)A_{ik}
    \\
    &\quad
    - 2 \Ric(\nu,\nu) A_{kj} A^k_j
    + 2 \Ric(e_k,e_l) A^k_i A^l_j
    + 2 \Ric(e_k,\nu) \nabla_{e_i}A^k_j
    \\
    &=
    -2 \Ric(\nu,\nu)A_i^k A_{kj}
    + 2 \Ric(e_k,e_l)A^k_i A^l_j
    + O_{L^2}(R^{-1})
    \\
    &=
    -\half  H^2 \Ric(\nu,\nu) \gamma_{ij}
    + \half H^2 T_{ij}
    + O_{L^2}(R^{-1}).
  \end{split}
\end{equation}
Here $T_{ij} = \Ric(e_i,e_j)$ denotes the tangential projection of the
Ricci-Tensor. The last step uses Lemma~\ref{thm:acirc_linfty} to
discard the terms containing $\Acirc$ into the error term.

Taking the trace in equation~\eqref{eq:5} yields that
\begin{equation}
  \label{eq:1}
  \Delta \Ric(\nu,\nu)
  =
  -\tfrac32 H^2 \Ric(\nu,\nu) + \half H^2 \Scal + O_{L^2}(R^{-1}).
\end{equation}
We thus infer that the trace free part of the Hessian of
$\Ric(\nu,\nu)$ is given by
\begin{equation}
  \label{eq:11}
  \begin{split}
    \big(\nabla^2\Ric(\nu,\nu)\big)^\circ_{ij}
    &=
    \nabla_{i,j}^2\Ric(\nu,\nu) - \half \Delta \Ric(\nu,\nu)\gamma_{ij}
    \\
    &=
    \half H^2
    \big( \half \Ric(\nu,\nu)\gamma_{ij} - \half \Scal \gamma_{ij} +
    T_{ij}\big)
    + O(R^{-1})
    \\
    &=
    \half H^2 \Tcirc_{ij}
    + O_{L^2}(R^{-1}).
  \end{split}
\end{equation}
Here we used that
\begin{equation*}
  \Tcirc_{ij}
  =
  T_{ij} - \half \tr T \gamma_{ij}
  =
  T_{ij} + \half \Ric(\nu,\nu)\gamma_{ij} - \half \Scal \gamma_{ij}.
\end{equation*}
\subsection{The covariant derivative of $\omega$}
\label{sec:covar-deriv-omega}
In a calculation similar to equation~\eqref{eq:5}, we derive
\begin{equation}
  \label{eq:12}
  \begin{split}
    \nabla_i \omega_j
    &=
    (\nabM_i\Ric)(\nu,e_j) + A_i^k \Ric(e_k,e_j) - A_{ij}
    \Ric(\nu,\nu)
    \\
    &=
    \half H ( T_{ij} - \Ric(\nu,\nu)\gamma_{ij}) + O_{L^\infty}(1).
  \end{split}
\end{equation}
Taking the trace yields
\begin{equation}
  \label{eq:13}
  \div\omega
  =
  \half H \Scal - \tfrac32 H \Ric(\nu,\nu)
  + O_{L^\infty}(1).
\end{equation}
Later we will also use the following combination
\begin{equation}
  \label{eq:14}
  \begin{split}
    2 \nabla_i \omega_j - \div\omega \gamma_{ij}
    &=
    H( T_{ij} + \half \Ric(\nu,\nu) \gamma_{ij} - \half \Scal
    \gamma_{ij})
    + O_{L^\infty}(1)
    \\
    &=
    H \Tcirc_{ij} + O_{L^\infty}(1).
  \end{split}
\end{equation}
\subsection{The Laplacian of $\protect\Tcirc$}
\label{sec:laplacian-tcirc}
First calculate the Hessian of $T$. Neglecting the lower order
terms yields
\begin{equation*}
  \begin{split}
    \nabla^2_{k,l} T_{ij}
    &=
    \tfrac14 H^2 \big(
    \gamma_{ki}\gamma_{lj}\Ric(\nu,\nu)
    +\gamma_{kj}\gamma_{li}\Ric(\nu,\nu)
    \\
    &\qquad\qquad
    - \gamma_{ki}\Ric(e_l,e_j)
    - \gamma_{kj}\Ric(e_l,e_i)\big)
    + O_{L^2}(R^{-1}).
  \end{split}
\end{equation*}
Taking the trace gives
\begin{equation*}
  \Delta T_{ij}
  =
  \half H^2 (\Ric(\nu,\nu) \gamma_{ij} - T_{ij}) + O_{L^2}(R^{-1}).
\end{equation*}
Thus we can calculate further
\begin{equation*}
  \Delta \Tcirc_{ij}
  =
  \Delta T_{ij}
  - \half \Delta \big(\Scal - \Ric(\nu,\nu)\big) \gamma_{ij}.
\end{equation*}
In view of the fact that
\begin{equation*}
  \Delta \Scal
  =
  \lapM \Scal - \nabM^2_{\nu,\nu}\Scal + H g(\nabM\Scal,\nu)
  =
  O_{L^2}(R^{-1}),
\end{equation*}
and the expression for $\Delta\Ric(\nu,\nu)$ in \eqref{eq:1}, we infer
that
\begin{equation}
  \label{eq:15}
  \Delta \Tcirc_{ij}
  =
  -\half H^2 \Tcirc_{ij} +O_{L^2}(R^{-1}).
\end{equation}


%
\section{Expansion of the curvature}
\label{sec:estimates}
In this section we consider the crucial geometric quantities on
$\Sigma$ as in section~\ref{sec:calculations} and derive the top
order deviations from their Euclidean value.
\subsection{Curvature corrections to $H^2$}
\label{sec:curv-corr-h2}
Combining equations~\eqref{eq:1} and~\eqref{eq:2} with the curvature
estimates we infer that
\begin{equation}
  \label{eq:3}
  \begin{split}
    &\Delta \big( \half H^2 - \tfrac23 \Ric(\nu,\nu) \big)
    \\
    &\quad=
    H \Delta H + |\nabla H|^2 - \tfrac23 \Delta\Ric(\nu,\nu)
    \\
    &\quad=
    - H^2 \Ric(\nu,\nu)  - H^2 \lambda
    - \tfrac23\big( - \tfrac32 H^2 \Ric(\nu,\nu) + \half H^2 \Scal \big)
    + O(R^{-1})
    \\
    &\quad=
    - H^2 \big( \lambda + \Scal\big) + O_{L^2}(R^{-1})
    \\
    &\quad=
    O_{L^2}(R^{-1}).
  \end{split}
\end{equation}
This identity leads to the following estimate.
\begin{proposition}
  \label{thm:H2estimate}
  Let $(M,g)$ be of $C_B$-bounded geometry. Then there exist constants
  $a_0\in(0,\infty)$ and $C$ depending only on $C_B$ such that for
  every surface $\Sigma$ satisfying~\eqref{eq:2}, with $|\Sigma|\leq
  a_0$ and $\CW(\Sigma)\leq 4\pi+a_0$ we have the estimate
  \begin{equation*}
    \big\| \half H^2 - 8\pi|\Sigma|^{-1} - \tfrac{2}{3}\Ric(\nu,\nu) +
    \tfrac59\Scal(0)\big\|_{L^\infty}
    \leq
    CR(\Sigma).
  \end{equation*}
\end{proposition}
\begin{proof}
  Let $w = \half H^2 - \tfrac23\Ric(\nu,\nu)$. The Bochner identity
  from Lemma~\ref{thm:bochner} implies that
  \begin{equation*}
    \int_\Sigma 2|\nabla^2 w|^2 + H^2 |\nabla w|^2 \dmu
    \leq
    \int_\Sigma |\Delta w|^2 + \big(\Ric(\nu,\nu) - \Scal +
    |\Acirc|^2 \big) |\nabla w|^2 \dmu.
  \end{equation*}
  Note that $\Ric$ and $\Scal$ are bounded by a constant, that
  $\|\Acirc\|_{L^\infty} \leq CR(\Sigma)$ by lemma~\ref{thm:acirc_linfty}, and
  that $H \geq C^{-1} R(\Sigma)^{-1}$ if $a_0$ is chosen small enough. If
  necessary we can decrease $a_0$ further so that the gradient term on the right can be
  absorbed to the left. In view of equation~\eqref{eq:3}, this yields
  \begin{equation*}
    \|\nabla^2 w \|_{L^2} \leq C
    \quad\text{and}\quad
    \|\nabla w\|_{L^2} \leq CR(\Sigma).
  \end{equation*}
  Consequently, the Poincar\'e inequality implies the estimate
  \begin{equation*}
    \|w - \bar w\|_{L^2} \leq CR(\Sigma)^2.
  \end{equation*}
  Plugging this into the estimate from
  lemma~\ref{thm:linfty-embedding}, we infer that
  \begin{equation}
    \label{eq:4}
    \|w-\bar w\|_{L^\infty}
    \leq
    CR(\Sigma).
  \end{equation}
  It remains to calculate $\bar w$. To this end recall
  \cite[Theorem 5.1]{lamm-metzger:2010}. This implies that
  \begin{equation*}
    \left| \int_\Sigma \half H^2 \dmu - 8\pi + \frac{|\Sigma|}{3}
      \Scal(0) \right|
    \leq
    CR(\Sigma)^3.
  \end{equation*}
  From \cite[Lemma 3.3]{lamm-metzger:2010} it follows that in addition
  \begin{equation*}
    \left| \int_\Sigma \Ric(\nu,\nu)\dmu - \frac{|\Sigma|}{3}\Scal(0)
    \right|
    \leq
    CR(\Sigma)^3.
  \end{equation*}
  In combination, this implies that for $\bar w = |\Sigma|^{-1}
  \int_\Sigma w \dmu$ we have.
  \begin{equation*}
    \big| \bar w - 8\pi + \tfrac 59\Scal(0) \big|
    \leq
    CR.
  \end{equation*}
  In view of~\eqref{eq:4} this yields the claim.
\end{proof}
\begin{remark}
  Note that this is \emph{not} the expansion of $H^2$ on geodesic
  spheres, which can be found in \cite[Lemma 2.4]{pacard-xu:2009} for
  example. This is due to the fact that geodesic spheres do not
  satisfy~\eqref{eq:2} on the order on which we do these
  calculations. In other words, if a surface satisfies~\eqref{eq:2},
  then its shape differs from that of a geodesic sphere in a way
  visible in the lower order correction terms of the mean curvature.
\end{remark}
\subsection{Curvature corrections for $H$ and its derivatives}
By a slight variation of terms, one can also derive estimates for $H$
instead of $\half H^2$. Alternatively one can proceed as
follows. Recall that for functions $f,g\in C^\infty(\Sigma)$ with
$g\neq 0$ we have the identity
\begin{equation*}
  \nabla^2_{i,j} \tfrac{u}{v}
  =
  - v^{-2}\big(\nabla_i v \nabla_j u + \nabla_i u \nabla_j v\big)
  + v^{-1}\nabla^2_{i,j} u - u v^{-2} \nabla^2_{i,j}v
  + 2 v^{-3} u \nabla_iv \nabla_j v.
\end{equation*}
Using the a priori estimates for $H$, $\nabla H$ and $\nabla^2 H$ as
before, we find that
\begin{equation*}
  \nabla^2_{i,j} \big(H^{-1}\Ric(\nu,\nu)\big)
  =
  H^{-1} \nabla^2_{i,j}\Ric(\nu,\nu) + O_{L^2}(1),
\end{equation*}
so that equation~\eqref{eq:5} yields
\begin{equation*}
  \nabla^2_{i,j} \big(H^{-1}\Ric(\nu,\nu)\big)
  =
  -\half H \Ric(\nu,\nu)\gamma_{ij} + \half H T_{ij} + O_{L^2}(1).
\end{equation*}
Splitting into trace part and trace-free part we get
\begin{equation}
  \label{eq:6}
  \Delta \big(H^{-1}\Ric(\nu,\nu)\big)
  =
  -\tfrac32 H \Ric(\nu,\nu) + \half H\Scal + O_{L^2}(1)
\end{equation}
and
\begin{equation}
  \label{eq:7}
  \big[\nabla^2 \big(H^{-1}\Ric(\nu,\nu)\big)\big]^\circ_{ij}
  =
  \half H \Tcirc_{ij} + O_{L^2}(1).
\end{equation}
Let $v := H - \tfrac23 H^{-1}\Ric(\nu,\nu)$. Combining
equations~\eqref{eq:6} and~\eqref{eq:2}, with the estimate from
theorem~\ref{thm:lambda_estimate} as in section~\ref{sec:curv-corr-h2}
we find that
\begin{equation*}
  \Delta v  =  O_{L^2}(1).
\end{equation*}
Arguing as before, the Bochner identity implies:
\begin{equation}
  \label{eq:9}
  \| \nabla^2 v \|_{L^2} \leq CR
  \quad
  \text{and}
  \quad
  \|\nabla v \|_{L^2} \leq CR^2.
\end{equation}
These considerations imply the following estimate.
\begin{proposition}
  \label{thm:nabla2Hcirc}
  Assume that $(M,g)$ and $\Sigma$ are as in
  Proposition~\ref{thm:H2estimate}. Then 
  \begin{equation*}
    \| (\nabla^2 H)^\circ - \tfrac13 H \Tcirc\|_{L^2}
    \leq
    CR(\Sigma)
    \quad
    \text{and}
    \quad
    \| \nabla H - \tfrac23\omega \|_{L^2}
    \leq
    CR(\Sigma)^2. 
  \end{equation*}
\end{proposition}
\begin{proof}
  The proof follows directly from the estimates~\eqref{eq:9} in
  combination with formulas~\eqref{eq:8} and~\eqref{eq:11}.
\end{proof}
\subsection{Curvature corrections for $\protect\Acirc$}
To estimate the corrections of the curvature to $\Acirc$ recall the
Simons-Identity on $\Sigma$ as in equation~\eqref{eq:24}. In view of
the a priori estimates from theorem~\ref{thm:apriori} and the
conventions in in section~\ref{sec:calculations}, on surfaces as in
theorem~\ref{thm:apriori} this yields
\begin{equation}
  \label{eq:16}
  \Delta \Acirc_{ij}
  =
  (\nabla^2 H)^0_{ij}
  + \half H^2 \Acirc_{ij}
  + 2 \nabla_i \omega_j - \div \omega\gamma_{ij}
  + O_{L^\infty}(R).
\end{equation}
In view of proposition~\ref{thm:nabla2Hcirc} and
equation~\eqref{eq:14} we thus infer
\begin{equation*}
  \Delta \Acirc_{ij}
  =
  \tfrac43 H \Tcirc_{ij} + \half H^2 \Acirc_{ij} + O_{L^2}(1).
\end{equation*}
In view of the a priori estimates and equation~\eqref{eq:15} we find
that
\begin{equation*}
  \Delta (H^{-1} \Tcirc)_{ij} = -\half H \Tcirc_{ij} + O_{L^2}(1),
\end{equation*}
so that the tensor
\begin{equation*}
  S_{ij} := \Acirc_{ij} + \frac43 H^{-1}\Tcirc_{ij}
\end{equation*}
satisfies
\begin{equation}
  \label{eq:17}
  \Delta S_{ij} = \half H^2 S_{ij} + O_{L^2}(1).
\end{equation}
Multiplying~\eqref{eq:17} by $S^{ij}$ and integrating by parts implies
\begin{equation*}
  \int_\Sigma |\nabla S|^2 + \half H^2 |S|^2 \dmu
  \leq
  \int_\Sigma |S| \dmu
  \leq
  \frac14 \int_\Sigma H^2|S|^2 \dmu + \int_\Sigma H^{-2} \dmu.  
\end{equation*}
Absorbing the first term on the right to the left yields the following
estimate.
\begin{proposition}
  \label{thm:acirc_correction}
  Assume that $(M,g)$ and $\Sigma$ are as in
  Proposition~\ref{thm:H2estimate}. Then 
  \begin{equation*}
    \| \Acirc + \tfrac43 H^{-1}\Tcirc\|_{L^2}
    \leq
    CR(\Sigma)^3
    \quad\text{and}\quad
    \| \nabla \Acirc + \tfrac43 H^{-1}\nabla\Tcirc\|_{L^2}
    \leq
    CR(\Sigma)^2.
  \end{equation*}
\end{proposition}


%
\section{Expansion of the metric}
\label{sec:expans-metr-deta}
It is well known\footnote{We use the convention for the curvature
  tensor from \cite[Section 2]{lamm-metzger-schulze:2009} that is 
  $\Rm_{\al\be\ga\nu}= \la(\nabla_\al\nabla_\be
  -\nabla_\be\nabla_\al)\partial_\ga,\partial_\nu\ra$
  and $\Rm^\nu_{\al\be\ga}= g^{\nu\mu}\Rm_{\al\be\ga\mu}$.}  that the
metric in normal coordinates has the expansion
\begin{equation*}
  g_{\al\be}(y) = \delta_{\al\be} + \tfrac13 \Rm_{\al\mu\be\nu} y^\mu y^\nu + O(|y|^2).
\end{equation*}
Here we denote $\Rm_{\al\mu\be\nu} = \Rm_{\al\mu\be\nu}(0)$ and all other
curvature quantities are evaluated at $0$ as well. From this we
calculate that
\begin{equation*}
  g_{\al\be,\mu\nu} (0) = \tfrac13 (\Rm_{\al\mu\be\nu} + \Rm_{\al\nu\be\mu}).
\end{equation*}
The Christoffel symbols thus satisfy
\begin{equation*}
  \begin{split}
    \Gamma_{\al\be,\ga}^\nu(0) &= \tfrac12 \delta^{\nu\mu}\big( g_{\al\mu,\be\ga} + g_{\be\mu,\al\ga} -
    g_{\al\be,\mu\ga})
    \\
    &= \tfrac16 \delta^{\nu\mu} (  \Rm_{\al\be\mu\ga} + \Rm_{\al\ga\mu\be} + \Rm_{\be\al\mu\ga} +
    \Rm_{\be\ga\mu\al} - \Rm_{\al\mu\be\ga} - \Rm_{\al\ga\be\mu})
    \\
    &= \tfrac13 \delta^{\nu\mu}(\Rm_{\al\ga\mu\be} + \Rm_{\be\ga\mu\al})
    \\
    &= -\tfrac13 \delta^{\nu\mu}(\Rm_{\al\ga\be\mu} + \Rm_{\be\ga\al\mu})
    = -\tfrac13 (\Rm^\nu_{\al\ga\be}  + \Rm^\nu_{\be\ga\al}).
  \end{split}
\end{equation*}
From this we get that in normal coordinates
\begin{equation*}
  \Gamma_{\al\be}^\nu(y) =\Gamma_{\al\be}^\nu(0) +
  \Gamma_{\al\be,\ga}^\nu(0)y^\ga +O(|y|^2)
  = -\tfrac13(\Rm^\nu_{\al\ga\be}  + \Rm^\nu_{\be\ga\al})y^\ga +O(|y|^2).
\end{equation*}
This implies that for a constant vector field
$b=b^\nu\del_\nu\in\IR^3$ we have
\begin{equation}
  \label{eq:nabla_b_expansion}
  \nabla_\al b^\nu
  = \del_\al b^\nu + \Gamma^\nu_{\al\be} b^\be
  = -\tfrac13(\Rm^\nu_{\al\ga\be}(0)  + \Rm^\nu_{\be\ga\al}(0))y^\ga b^\be +O(|y|^2)
\end{equation}
and
\begin{equation}
  \label{eq:div_b_expansion}
  \div b = \nabla_\al b^\al = -\tfrac13 \Ric_{\be\ga}(0) y^\ga b^\be +O(|y|^2).
\end{equation}
Terms like these will show up in the position estimates. Here we
explicitly included the point at which to evaluate the curvature for
later reference.


%
\section{The position estimates revisited}
\label{sec:position}
The basic idea of the position estimates for small area constrained
Willmore surfaces in \cite{lamm-metzger:2010} is to test the
Euler-Lagrange-Equation 
\begin{equation}
  \label{eq:19}
  \delta_f \CW(\Sigma) = \lambda \int_\Sigma fH \dmu
\end{equation}
with the function $f= H^{-1}g(b,\nu)$. The main result of
\cite{lamm-metzger:2010} is the estimate
\begin{equation}
  \label{eq:18}
  |\nabla\Scal(p)| \leq CR(\Sigma)
\end{equation}
where $p$ is a point with $\dist(p,\Sigma) \leq \diam(\Sigma)$ and $C$
is a constant depending only on $C_B$. With this choice of point, we
have that $r(x)=\dist(p,x)$ for $x\in\Sigma$ is comparable to the area
radius $R(\Sigma)$.

To improve this estimate further we have to carefully choose the
center point $p$ of the above coordinates. The main result of the
paper in this section is:
\begin{theorem}
  \label{thm:position}
  Let $(M,g)$ be a 3-manifold with $C_B$-bounded geometry. Then there
  exist constants $a_0\in(0,\infty)$ and $C\in(0,\infty)$ depending
  only on $C_B$ with the following property. Let $\Sigma\subset M$ be
  a surface satisfying the Euler-Lagrange equation~\eqref{eq:2} for
  some $\lambda\in\IR$, with $|\Sigma|\leq a_0$, and
  $\CW(\Sigma)\leq 4\pi+ a_0$. Then there exists a point $p_0\in M$
  such that
  \begin{enumerate}[i)]
  \item $|\dist(p_0,x)-R(\Sigma)| \leq CR(\Sigma)^3$ for all
    $x\in\Sigma$, 
  \item with respect to normal coordinates $\psi:\CB_\rho(p_0)
    \to B_\rho(0)\subset \IR^3$ centered at $p_0$ we have
    \begin{equation*}
      \int_{\psi(\Sigma)} y^\alpha \dmu_g(y) = 0,
    \end{equation*}
  \item and  $|\nabla\Scal(p_0)| \leq CR(\Sigma)^2$.
  \end{enumerate}
\end{theorem}
\begin{remark}
  \label{rem:sphere-approximation}
  Note that appealing to theorem~\ref{thm:delellis-muller},
  lemma~\ref{thm:euclidean-umbilic} and estimate~\ref{thm:apriori}, we
  automatically have that $\Sigma$ is $W^{2,2}$-close to the geodesic
  sphere of radius $R$ around $p_0$ in the following sense. Denote by
  $h_R : \IR^3\to\IR^3 : y \mapsto \frac{y}{R}$ the scaling vector
  field. Let
  $\Sigma_R := h_R(\psi(\Sigma) \subset B_{\frac{\rho}{R}}(0)$.  Then
  there is a map $F : S^2 \to \Sigma_R$, conformal with respect to
  metric on $\Sigma_R$ induced by the Euclidean metric such that
  \begin{equation*}
    \|F\|_{W^{2,2}(S^2)} \leq C R(\Sigma)^2.
  \end{equation*}
\end{remark}
For the proof of Theorem~\ref{thm:position} assume that $\Sigma$ is as
in the statement of Theorem~\ref{thm:position} and that $a_0$ is
chosen so small that all the estimates from
sections~\ref{sec:preliminaries} to section~\ref{sec:expans-metr-deta}
are applicable. In particular, Lemma~\ref{thm:geometric-center} gives
a point $p_0\in M$ such that
$|\dist(p_0,x)-R(\Sigma)| \leq CR(\Sigma)^3$ for all $x\in\Sigma$ and
such that if $\psi: \CB_\rho(p_0) \to B_\rho(0)\subset \IR^3$ are
normal coordinates at $p_0$ then $\int_{\psi(\Sigma)} y \dmu_g =
0$. Then the first two assertions of the Theorem directly follow. The
estimate for $|\nabla\Scal(p_0)|$ follows from the calculations in the
remainder of this section. All these calculations are done in the
normal coordinates centered at $p_0$.

Recall the splitting
\begin{equation}
  \label{eq:32}
  \delta_f \CW(\Sigma) = \delta_f\CU(\Sigma) + \delta_f\CV(\Sigma)
\end{equation}
that was used with~\eqref{eq:19} for the test function
$f=H^{-1}g(b,\nu)$. Here $b\in\IR^3$ is a constant vector in the
normal coordinate neighborhood. The computations below use the same $f$
and the same splitting.

\subsection{The right hand side}
For $f=H^{-1}g(b,\nu)$ we have that
\begin{equation*}
  \int_\Sigma fH \dmu = \int_\Sigma g(b,\nu)\dmu = \int_\Omega \div b \dvol
\end{equation*}
where $\Omega$ is the region enclosed by $\Sigma$. In the integral on
the right, we replace the volume form of $g$ by the Euclidean volume
form of the normal coordinates at $p_0$ and obtain an error of the form
\begin{equation*}
  \left| \int_\Omega \div b \dvol - \int_\Omega \div b \dvol^E \right|
  \leq \Vol(\Omega) \sup_\Omega|\nabla b| R^2 \leq C|\Sigma|^3
\end{equation*}
since $|\nabla b| = O(R)$, $|\dvol - \dvol^E| = O(R^2)\dvol^E$ and
$\Vol(\Omega) = O(R^3)$ by~\cite[Eq.~(4.7)]{lamm-metzger:2010}. As
usual we abbreviate $R=R(\Sigma)$. At this point we do not care about
errors of the order $O(R^5)$ but want to compute the top order term
which is $O(R^4)$. In section~\ref{sec:expans-metr-deta} we computed
that
\begin{equation*}
  \div b = -\tfrac13\Ric_{\alpha\beta} y^\alpha b^\beta + O(R^2).
\end{equation*}
Note that the volume integral of the error term is $O(R^5)$, and that
$\Ric$ here is evaluated at $p_0$, the origin of the normal
coordinates $y$. We thus get
\begin{equation*}
  \begin{split}
    \int_\Omega \div b \dmu
    &=  -\tfrac13\Ric_{\alpha\beta} b^\beta
    \int_\Omega y^\alpha \dmu^E +O(R^5)
    \\
    &=
    -\tfrac16\Ric_{\alpha\beta} b^\beta \int_\Sigma |y|^2 (\nu^E)^\alpha
    \dmu^E
    + O(R^5).
  \end{split}
\end{equation*}
Here $\nu_E$ denotes the normal to $\Sigma$ with respect to the
Euclidean metric in our coordinates. From Corollary~\ref{thm:even-odd}
with $k=1$, $\rho_1= \rho_2=\frac{y^\beta}{R}$ and $\rho_3 =
(\nu^E)^\alpha$ it follows that
\begin{equation*}
  \int_\Sigma (y^\beta)^2 (\nu^E)^\alpha \dmu^E \leq CR^5
\end{equation*}
and after summation over $\beta$ we arrive at
\begin{equation*}
  \left| \int_\Omega \div b \dmu \right|
  \leq
  C R^5
\end{equation*}
or, in combination with the estimates above, with~\eqref{eq:19} and
using Theorem~\ref{thm:lambda_estimate} this gives
\begin{equation}
  \label{eq:27}
  |\delta_f\CW(\Sigma)| \leq CR^5.
\end{equation}
Note that this improves the estimate from \cite{lamm-metzger:2010} by
one power of $R$.

\subsection{The variation of $\CU(\Sigma)$}
\label{sec:variation-acirc}
From \cite{lamm-metzger:2010} we have
\begin{equation}
  \label{eq:20}
  \delta_f \CU(\Sigma) =
  - \int_\Sigma
  2 \la \Acirc,\nabla^2 f \ra
  + 2 f\la\Acirc,\Tcirc\ra
  + fH |\Acirc|^2 \dmu
\end{equation}
and that for our choice $f= H^{-1}g(b,\nu)$ we have
\begin{equation*}
  \begin{split}
    \nabla^2_{ij} f &=
    - A_i^k A_{jk} f + H^{-1}g(\nabla_ib,e_k)A^{kj} - H^{-2}\nabla_i H
    g(b,e_k) A^k_j
    \\
    &\phantom{=}
    +\nabla_i\big(H^{-1}g(\nabla_jb,\nu) - H^{-2} \nabla_j H g(b,\nu)\big)
  \end{split}
\end{equation*}
Since
\begin{equation*}
  A_i^k A_{jk} = \Acirc_i^k \Acirc_{jk} + H \Acirc_{ij} + \tfrac14 H^2 g_{ij}
\end{equation*}
and since the first and the last term give zero when contracted with
$\Acirc$ we get that
\begin{equation*}
  \begin{split}
    \int_\Sigma \la \nabla^2 f, \Acirc\ra\dmu
    &= \int_\Sigma \div \Acirc^j \big(H^{-2} \nabla_jH g(b,\nu) - H^{-1}
    g(\nabla_j b,\nu)\big)
    + f H |\Acirc|^2
    \\
    &\phantom{= \int_\Sigma}
    -H^{-2}\nabla_iH g(b,e_k) A^k_j\Acirc^{ij}
    +H^{-1}g(\nabla_ib,e_k)A^k_j \Acirc^{ij}
    \dmu.
  \end{split}
\end{equation*}
Plugging into~\eqref{eq:20} we get
that
\begin{equation*}
  \begin{split}
    \delta_f \CU(\Sigma)
    &= 2\int_\Sigma \div \Acirc^j \big( H^{-1}
    g(\nabla_j b,\nu) - H^{-2} \nabla_jH g(b,\nu) \big)
    -\tfrac32 f H |\Acirc|^2
    \\
    &\phantom{=\int_\Sigma}
    + H^{-2}\nabla_iH g(b,e_k) A^k_j\Acirc^{ij}
    - H^{-1}g(\nabla_ib,e_k)A^k_j \Acirc^{ij}
    - f\la\Acirc,\Tcirc\ra
    \dmu.
  \end{split}
\end{equation*}
We shall only keep the top order parts of the first two terms in the
second line. In view of the $L^\infty$ estimates for $\Acirc$,
$H^{-1}$ and the $L^2$-estimates for $\nabla H$, we have that
\begin{equation}
  \label{eq:21}
  \begin{split}
    \delta_f \CU(\Sigma)
    &= \int_\Sigma 2\div \Acirc^j \big( H^{-1}
    g(\nabla_j b,\nu) - H^{-2} \nabla_jH g(b,\nu) \big)
    - 3 f H |\Acirc|^2
    \\
    &\phantom{=\int_\Sigma}
    +  H^{-1}\nabla_iH g(b,e_j) \Acirc^{ij}
    -  g(\nabla_ib,e_j) \Acirc^{ij}
    - 2 f\la\Acirc,\Tcirc\ra
    \dmu
    +O(R^5).
  \end{split}
\end{equation}
In view of the estimates in section~\ref{sec:estimates} all the above
terms can be replaced with their highest order parts. The error terms
are then of order $O(R^5)$ or better. To be specific, we recall that
to top order
\begin{equation*}
  H^{-1} \approx \tfrac R2,\quad
  \Acirc \approx -\tfrac43H^{-1}\Tcirc,\quad
  \nabla H \approx \tfrac23\omega,\quad\text{and}\quad
  \div\Acirc = \tfrac43\omega.
\end{equation*}
This yields that
\begin{equation}
  \label{eq:22}
  \begin{split}
    \delta_f \CU(\Sigma)
    &=
    \frac 23\int_\Sigma
    2 R\omega_j g(\nabla_jb, \nu)
    - \tfrac 23 R^2 |\omega|^2g(b,\nu)
    -  R^2 g(b,\nu) |\Tcirc|^2
    \\
    &\phantom{=\frac 23\int_\Sigma}
    - \tfrac 13 R^2 \omega_i g(b,e_j) \Tcirc^{ij}
    + R g(\nabla_ib,e_j)\Tcirc^{ij} \dmu
    +O(R^5)
  \end{split}
\end{equation}
Note that all the previous terms can be expanded into integrals that
can individually be estimated using
Corollary~\ref{thm:even-odd}. Consider for example the first term on
the right of~\eqref{eq:22}:
\begin{equation*}
  \begin{split}
    &\frac{4R}{3}\int_\Sigma \sum_{j=1}^2 \omega(e_j)
    g(\nabla_{e_j}b,\nu)\dmu\\
    &=
    \frac{4R}{3} \int_\Sigma \sum_{\be=1}^3 (\Ric_{\al\be}\nu^\al
    g_{\eta\mu}\nabla_\beta b^\eta\nu^\mu) - \Ric_{\al\be}\nu^\al\nu^\be
    g_{\eta\mu}\nabla_\kappa b^\eta\nu^\mu\nu^\kappa \dmu.
  \end{split}
\end{equation*}
After replacing $\nabla b$ using the
expansion~\eqref{eq:nabla_b_expansion}, $g_{\eta\mu} =
\delta_{\eta\mu} + O(R^2)$, $\Ric = \Ric(p_0) + O(R)$ and noting that
the resulting error terms are of order $O(R^5)$ we can use
Corollary~\ref{thm:even-odd} to see that the whole term is
$O(R^5)$. Inspecting the other terms of~\eqref{eq:22} shows that they
can be treated similarly. Indeed all the tangential contractions in
these terms can be resolved as above and the remaining terms are
products of an odd number of factors $\nu$ or $y/R$. To show the
pattern note that
\begin{equation*}
  \begin{split}
    |\omega|^2 &=  |\Ric(\nu,\cdot)^T|^2 = |\Ric(\nu,\cdot)|^2 -
    \Ric(\nu,\nu)^2
    \\
    &=
    \sum_{\be=1}^3 g_{\be\kappa}\Ric_{\al\be}\Ric{\eta\kappa}\nu^\al\nu^\eta
    - (\Ric_{\al\be}\nu^\al\nu^\be)^2
  \end{split}
\end{equation*}
Both terms on the right have an even number of factors $\nu$ so that
multiplied with $g(b,\nu)$ in the second term on the right
of~\eqref{eq:22} yields an odd number. The third term can be treated
by computing with $\tau = \tr T = \Scal -\Ric(\nu,\nu)$ that
\begin{equation*}
  \begin{split}
    |\Tcirc|^2
    &= |T|^2 -\tfrac12\tau^2 \\
    &= |\Ric|^2 -2|\omega|^2 - \Ric(\nu,\nu)^2 - \tfrac 12\Scal^2
    -\Scal\Ric(\nu,\nu)
    -\tfrac12\Ric(\nu,\nu)^2.
  \end{split}
\end{equation*}
Note that all terms on the right contain an even number of factors
$\nu$, so the third term in~\eqref{eq:22} is also done.
The remaining two terms have a similar structure. We infer the
estimate
\begin{equation}
  \label{eq:23}
  |\delta_f\CU(\Sigma)| \leq CR^5
\end{equation}
for the particular choice of $f$ above.

\subsection{The variation of $\CV(\Sigma)$}
From~\cite[Section 4.3]{lamm-metzger:2010} we get that for the given
choice of $f$ we have
\begin{equation}
  \label{eq:28}
  \begin{split}
    \delta_f \CV(\Sigma) &= \int_\Sigma - G(b,\nu) - \half g(b,\nu)\Scal
    + 2f\la\Acirc,G^T\ra \\
    &\quad -2\omega(e_i)\big(H^{-1}g(\nabla_{e_i}b,\nu) +
    H^{-1}\Acirc_i^jg(b,e_j) - H^{-2}\nabla H g(b,\nu)\big)\dmu.
  \end{split}
\end{equation}
As in section~\ref{sec:variation-acirc} we can estimate
\begin{equation*}
  \begin{split}
    &\left| \int_\Sigma 2f\la\Acirc,G^T\ra
    -2\omega(e_i)\big(H^{-1}g(\nabla_{e_i}b,\nu)
    + H^{-1}\Acirc_i^jg(b,e_j) - H^{-2}\nabla Hg(b,\nu)\big)\dmu
  \right|
  \\
  &\quad\leq CR^5.
  \end{split}
\end{equation*}
To see this, use $\nabla H = \frac{2}{3}\omega + O_{L^2}(R)$ and
$\Acirc=-\frac{4}{3}H^{-1}\Tcirc + O_{L^2}(R^2)$ from
Propositions~\ref{thm:nabla2Hcirc} and~\ref{thm:acirc_correction}. The
estimate then follows by inspection as in
section~\ref{sec:variation-acirc}.

Furthermore, as in~\cite[Section 4.3]{lamm-metzger:2010} let $X$ be
the vector field on $\CB_\rho(p_0)M$ such that $g(X,Y) = G(b,Y)$ for
all vector fields $Y$ on $\CB_\rho(p_0)$. Then
$\div_M X = \la G, \nabla b\ra$ since $G$ is divergence free. Let
$\Omega\subset \CB_\rho(p_0)$ enclosed by $\Sigma$ and recall
from~\cite[Section 4.3]{lamm-metzger:2010} that $\Vol(\Omega)\leq CR^3$. Compute
\begin{equation*}
  \begin{split}
    &\int_\Sigma G(b,\nu)\dmu = \int_\Omega \div_M X \dvol =
    \int_\Omega \la G,\nabla b\ra \dvol = \int_\Omega
    G_{\al\be}\nabla_\kappa b^\alpha g^{\beta\kappa} \dvol
    \\
    &= -\tfrac{1}{3}
    G_{\al\be}(p_0)\delta^{\be\kappa}(\Rm_{\kappa\eta\mu}^\al(p_0)+\Rm^\al_{\mu\eta\kappa}(p_0))b^\mu
    \int_\Omega y^\kappa \dvol^E +O(R^5).
  \end{split}
\end{equation*}
Here we used equation~\eqref{eq:nabla_b_expansion} in the last step
and replaced all curvature quantities by their values at $p_0$. Also
the integration is now with respect to the Euclidean volume form
$\dvol^E$. The value of the constant in front of the integral is not
important for the following. For $\kappa\in\{1,2,3\}$ consider the
vector field $Y:=\tfrac{1}{4}y^\kappa y$. Then $\div_E Y = y^\kappa$
and thus using Stokes in the first equality and
Corollary~\ref{thm:even-odd} in the estimate.
\begin{equation}
  \label{eq:31}
  \int_\Omega y^\kappa \dvol^E = \int_\Sigma y^\kappa \la y,\nu^E\ra_E
  \dmu^E = O(R^5).
\end{equation}
To treat the remaining term, we consider the
vector field $X= \Scal b$ as in~\cite[Section
4.3]{lamm-metzger:2010}. Then
\begin{equation*}
  \int_\Sigma g(b,\nu)\Scal \dmu = \int_\Omega \div_M X \dvol  =
  \int_\Omega g(b,\nabla \Scal)  + \Scal
  \div_M b \dvol
\end{equation*}
with $\Omega$ as above. Using $g= g^E + O(R^2)$, $\Scal = \Scal(p_0) + O(R)$, $\nabla_\beta
\Scal = \nabla_\beta\Scal (0) + \nabla^2_{\beta,\kappa}\Scal(0)
y^\kappa + O(R^2)$ and equation~\eqref{eq:div_b_expansion} for the
expansion of $\div b$, we get
\begin{equation*}
  \begin{aligned}
    &\int_\Sigma g(b,\nu)\Scal \dmu
    \\
    &=
    \int_\Omega g^E(b,\nabla\Scal(p_0)) \dvol^E
    +
    \big(b^\alpha \nabla_{\alpha\kappa}^2\Scal -
    \tfrac{1}{3} \Scal(p_0) \Ric_{\beta\kappa}(p_0) b^\beta\big)
    \int_\Omega y^\kappa \dvol^E + O(R^5).
  \end{aligned}
\end{equation*}
In view of~\eqref{eq:31} this gives
\begin{equation*}
  \int_\Sigma g(b,\nu)\Scal \dmu = \Vol(\Omega) g(b,\nabla\Scal(p_0))
  + O(R^5).
\end{equation*}
In combination with the above, we arrive at the estimate
\begin{equation}
  \label{eq:30}
  \left| \delta_f \CV(\Sigma) + \tfrac{1}{2} \Vol(\Omega)
    g(b,\nabla\Scal(p_0)) \right| \leq CR(\Sigma)^5.
\end{equation}

\subsection{The conclusion}
From the splitting~\eqref{eq:32},
estimates~\eqref{eq:27},~\eqref{eq:23}, and~\eqref{eq:30} we arrive at
\begin{equation*}
  \Vol(\Omega) |g(b,\nabla\Scal(p_0))| \leq CR(\Sigma)^5.
\end{equation*}
Since $b\in\IR^3$ is arbitrary and since  $\Vol(\Omega)\geq C^{-1}
R(\Sigma)^3$ by~\cite[Eq.~(4.7)]{lamm-metzger:2010} this gives the claimed estimate:
\begin{equation*}
  |\nabla\Scal(p_0)| \leq CR(\Sigma)^2.
\end{equation*}
This concludes the proof of Theorem~\ref{thm:position}.


%
\section{The proof of Corollary~\ref{thm:position-inside-intro}}
\label{sec:inside}
\begin{corollary}
  \label{thm:position-inside}
  Let $(M,g)$ be a compact three dimensional Riemannian manifold with
  $C_B$ bounded geometry. Let
  \begin{equation*}
    Z := \{x\in M \mid \nabla\Scal(x) = 0\}
  \end{equation*}
  and assume that the Hessian $\Hess \Scal (x)$ is non-degenerate for
  every $x\in Z$.
  
  Then there exists an $a_0$ depending only on $(M,g)$ such that for
  every surface $\Sigma$ that satisfies the Euler-Lagrange
  equation~\eqref{eq:2} for some $\lambda$, with $|\Sigma|\leq a_0$
  and $\CW(\Sigma)\leq 4\pi + a_0$ the region enclosed by $\Sigma$
  intersects $Z$ in a single point.
\end{corollary}

\begin{proof}
  Since $M$ is compact and all critical points of $\Scal$ are
  non-degenerate, the set $Z$ is discrete. Let
  \begin{equation*}
    \rho_0 := \frac{1}{2} \min \{ \dist(x,y)  \mid x\neq y \in Z\}.
  \end{equation*}
  For $\rho\in(0,\rho_0)$ let 
  \begin{equation*}
    Z_\rho := \{x\in M \mid \dist(x,Z) < \rho\}.
  \end{equation*}
  For $r\in(0,\infty)$ let
  \begin{equation*}
    G_r := \{x\in M \mid |\nabla\Scal(x)| \leq r\}.
  \end{equation*}
  By the compactness of $M$ and since $\Scal$ is a Morse function,
  there exist $r_0\in(0,\infty)$ and $c\in(0,\infty)$ such that for
  every $r\in (0,r_0)$ we have $G_r \subset Z_{cr}$.

  Let $a_0$ and $C$ be the constants from Theorem~\ref{thm:position}
  applied to $(M,g)$. By decreasing $a_0$, we can assume that in
  addition to the assertion of Theorem~\ref{thm:position}, we also
  have that $C R(\Sigma)^2 \leq r_0$ and that
  $\diam(\Sigma) < \rho_0$ whenever $\Sigma$ satisfies the assumption
  of this Lemma with the chosen $a_0$.

  Let $\Sigma$ be such a surface and let $R=R(\Sigma)$. Denote by
  $\Omega$ the open region enclosed by $\Sigma$. For $s\in(0,\infty)$ denote
  \begin{equation*}
    \Omega_s := \{ x\in\Omega \mid \dist(x,\Sigma) > s\}.
  \end{equation*}
  Then $\Omega_s$ is an open subset of $\Omega$. 

  Let $p_0$ be the point from Theorem~\ref{thm:position}. Then
  $|\dist(p_0,x) - R|\leq CR^3$ for all $x\in \Sigma$ and
  $|\nabla \Scal(p_0)| \leq C R^2$. The first estimate shows that
  $p_0 \in \Omega$, in fact $p_0 \in \Omega_{\frac{3}{4}R}$ if we
  choose $a_0$ sufficiently small.

  Let $g:= |\nabla\Scal(p_0)|$. Then $g\leq CR^2$ so that
  $p_0\in G_g \subset Z_{cg}$. This implies that there exists a point
  $p_1\in Z$ such that
  $p_1\in \Omega{\frac{3}{4}R - cg} \subset \Omega_{\frac{3}{4}R -
    CR^2}$. By choosing $a_0$ and thus $R$ smaller again, we can
  ensure that $\frac{3}{4}R - CR^2 \geq \frac{R}{2}$, so that
  $Z\cap \Omega_{\frac{R}{2}}\neq \emptyset$. Since
  $\diam(\Sigma)<\rho_0$ we know that
  $\Omega_{\frac{R}{2}}\subset \Omega\subset B_{\rho_0}(p_0)$ and by
  the choice of $\rho_0$ the ball $B_{\rho_0}(p_0)$ can intersect $Z$
  in at most one point.
\end{proof}
From the expansion of the Willmore energy in~\cite[Corollary
5.6]{lamm-metzger:2013} we know that the minimizers
$\Sigma^\text{min}_a$ from Theorem~\ref{thm:minimizers} concentrate
near the maxima of the scalar curvature of $M$. Thus a slight variant
of the proof of Corollary~\ref{thm:position-inside} yields the proof
of Corollary~\ref{thm:position-inside-minimizers}.


%
\bibliographystyle{abbrv}
\bibliography{../extern/references}
\end{document}